\documentclass{article}
\usepackage{amsmath,amssymb,amsthm,stmaryrd,makeidx}
\usepackage[all]{xy}

\DeclareMathOperator{\enm}{End}
\DeclareMathOperator{\ext}{Ext}

\DeclareMathOperator{\id}{id}

\DeclareMathOperator{\im}{im}

\DeclareMathOperator{\hmm}{Hom}

\newcommand{\cat}[1]{\mathbf{#1}}

\newcommand{\grmcat}[1]

\newcommand{\gl}{\mathsf{GL}}

\newcommand{\defm}[1]{\mathsf{Def}_{#1}}

\input xypic

\newtheorem{proposition}{Proposition}
\newtheorem{theorem}{Theorem}
\newtheorem{lemma}{Lemma}
\newtheorem{corollary}{Corollary}
\newtheorem{definition}{Definition}

\newtheorem{remark}{Remark}

\newtheorem{example}{Example}

\makeindex

\begin{document}
\author{Arvid Siqveland}
\title{Countably Generated Matrix Algebras}

\maketitle

\begin{abstract} 
We define the completion of an associative algebra $A$ in a set $M=\{M_1,\dots,M_r\}$ of $r$ right $A$-modules in such a way that if $\mathfrak a\subseteq A$ is an ideal in a commutative ring $A$ the completion $A$ in the (right) module $A/\mathfrak a$ is $\hat A^M\simeq \hat A^{\mathfrak a}.$ This works by defining $\hat A^M$ as a formal algebra determined up to a computation in a category called GMMP-algebras. From deformation theory we get that the computation results in a formal algebra which is the prorepresenting hull of the noncommutative deformation functor, and this hull is unique up to isomorphism.
\end{abstract}

\section{Introduction}
All vector spaces and algebras will be over an arbitrary field $k.$ All algebras will be associative and unital. We say left or right ideal when necessary, the word ideal is reserved for two-sided ideals. When $M$ is a right module over an associative ring $A,$ we let $\rho_M:A\rightarrow\enm_{\mathbb Z}(M)$ denote the structure morphism, that is $\rho(a)(m)=ma.$

In the article \cite{Siqveland2301} we constructed schemes by using only completions of rings in maximal ideals. The goal of this paper is to find an alternative definition of the completion in an ideal that can be generalized to the completion of an associative $k$-algebra $A$ in a set $M=\{M_1,\dots,M_r\}$ of $r$ simple right $A$-modules. 

We define a category $\cat L$ called GMMP-algebras with a functor $\hat A:\cat L\rightarrow\cat{Alg}_k.$ For the set $M=\{M_1,\dots,M_r\}$ of $r$ simple right $A$-modules we associate a GMMP-algebra $L(M).$ Then we define the completion of $A$ in $M$ as $$\hat A_M=\hat A(L(M)).$$

These completions are associative rings $\Hat A_M$ fitting in the diagram 

$$\xymatrix{k^r\ar[dr]_{\id}\ar[r]&\hat A_M\ar[d]^\pi\\&k^r,}$$

such that $\hat A_M=\underset{\underset n\leftarrow}\lim A/(\ker\pi)^n,$ justifying the word completions, and such that the following diagram commutes:

$$\xymatrix{A\ar[r]^-\iota\ar[dr]_{\oplus\rho_i}&\enm_{\hat A_M}(\hat A_M\otimes_{k^r}(\oplus_{i=1}^rM_i))\ar[d]\\&\oplus_{i=1}^r\enm_k(M_i,M_i)}$$ which justify that this is a generalization of the completion in a single maximal ideal $\mathfrak m\subseteq C$ of a commutative $k$-algebra $C$ such that $M_1=C/\mathfrak m\simeq k.$

\section{Linear Algebra}
Let $B_W=\{w_i\}_{i=1}^n$ be a basis for a vector space $W$ and let $W\overset\kappa\rightarrow V\rightarrow 0$ be a quotient such that $B_V=\{v_i=\kappa(w_i)|1\leq i\leq d\leq n\}$ is a basis for $V.$ Then we have a system of unique relations

$$\begin{matrix}
\alpha_{11}v_1&+\alpha_{12}v_2&+&\cdots&+&\alpha_{1d}v_d&-v_{d+1}&=0\\
\alpha_{21}v_1&+\alpha_{22}v_2&+&\cdots&+&\alpha_{2d}v_d&-v_{d+2}&=0\\
\vdots&&&&&&&\vdots\\
\alpha_{(n-d)1}v_1&+\alpha_{(n-d)2}v_2&+&\cdots&+&\alpha_{(n-d)d}v_d&-v_{n}&=0.
\end{matrix}$$

Bring this system of equations to its reduced row-echelon form. Up to renumbering of the basis $B_W$ this can be written in the form:

\begin{equation}\label{releq}
\begin{matrix}
v_1&&&&+\beta_{1,r+1}v_{r+1}&+&\cdots&+&\beta_{1,n}v_n&=0\\
&v_2&&&+\beta_{2,r+1}v_{r+1}&+&\cdots&+&\beta_{2,n}v_n&=0\\
&&\ddots&&&&&&\vdots&\\
&&&v_{r}&+\beta_{r,r+1}v_{r+1}&+&\cdots&+&\beta_{r,n}v_n&=0.
\end{matrix}\end{equation}

\begin{definition}\label{relmorphdef} We define the relation morphism $d_{VW}:W\rightarrow W$ as the linear transformation given by $$d_{VW}(w_j)=\begin{cases}w_j-\sum_{i=1}^{n-r}\beta_{j,r+i}w_{r+i},&1\leq j\leq r,\\
0,&r<j\leq n.
\end{cases}
$$
\end{definition}

Due to the minus sign, we have that the $r$ expressions in (\ref{releq}) can be written $$f_i=\sum_{j=1}^nw_i^\ast(w_j)$$ where $\{w_1,\dots,w_r\}$ is a basis for $W/\im d_{VW}.$ If another basis of  $W/\im d.$ is chosen, this amounts to a linear change of variables. It follows that the elements $\{f_1,\dots,f_r\}$ are linearly independent, and so we consider the subspace 
$F=(f_1,\dots,f_r)\subseteq W$ with basis $\{f_1,\dots,f_r).$ We also let $F:W\rightarrow W$ denote the map $F(w_i)=\begin{cases}f_i,&1\leq i\leq r,\\
0,&i>r\end{cases}.$ We then have the following.

\begin{lemma}\label{rellemma} With the notation above, there is an exact sequence 
$$W\overset{F}\rightarrow W\overset\kappa\rightarrow V\rightarrow 0,$$ in particular $V\simeq W/F.$
\end{lemma}

\begin{proof} This follows because $F$ maps to the defining relations of $V$ as quotient of $W.$
\end{proof}

When we generalize to the situation where $W$ is infinite dimensional, we will be in two different situations in the following. This leads to the concept of \emph{formal} vector spaces.

Let $B$ be a basis for a vector space $V$ over a field $k.$ Let $\widehat V_B$ be the direct product $$\widehat V_B=\prod_{b\in B}k\ni (\alpha_b)_{b\in B}, \alpha_b\in k, b\in B,$$ and notice that these are infinite products contrary to the direct sum $\sum_{b\in B}k=\sum_{b\in B}\alpha_bb$ in which only a finite number of the $\alpha_b$ are different from $0.$ It is clear that $\widehat V_B$ is a vector space under component-wise addition and scalar multiplication. As every vector space has a basis, this also holds for $\widehat V_B,$ but we see that if $B$ is infinite, then $B$ is a linearly independent set, but not a basis.

\begin{definition} The completion of a vector space $V$ over $k$ with respect to a basis $B$ is the vector space $\widehat V_B=\prod_{v\in B}k.$ A vector space which is isomorphic to a completion is called complete, or formal.
\end{definition}

\section{$k$-Algebras}\label{AlgebrasSection}

Let $V=\oplus_{i=1}^n kx_i$ be an $n$-dimensional vector space. Let $T^i(V)=V^{\otimes i}$ $i$ times for $i\geq 1$ and let $T^0(V)=k.$ 
The free associative algebra in $n$ variables over the field $k$ is the tensor-algebra 
$$F=\oplus_{i=0}^\infty T^i(V).$$ We use the notation
$$k\langle n\rangle=k\langle x_1,\dots,x_n\rangle\overset\rho\rightarrow k,$$ and we let $\mathfrak m=\ker\rho.$ With respect to the choice of base-point, we see that $F$ is augmented over $k$ by sending $x_i$ to $0,$ $1\leq i\leq n.$
Interpreting $\mathfrak m^0=k\langle n\rangle,$ we  can write the vector space $$k\langle n\rangle=\oplus_{i=0}^\infty \mathfrak m^i/\mathfrak m^{i+1}.$$
Each $n^i$-dimensional vector space $\mathfrak m^i/\mathfrak m^{i+1}$ has a basis consisting of all monomials of degree $i$ which corresponds bijectively to the set $B_i=\{w:|w|=i$\} of words in the alphabet with $n$ letters of length $i.$ 

We denote the free associative algebra generated by a countable arbitrary set $S$ over $k$ by $k\langle S\rangle=k\langle x_s;s\in S\rangle$ and we say that $k\langle\mathbb N\rangle$ is countably generated. We still have the augmentation $\rho:k\langle S\rangle\rightarrow k,\ \rho(x_s)=0,\ s\in S,$ and $\mathfrak m=\ker\rho.$

\begin{definition} Let $A$ be an associative $k$-algebra. A subset $S\subseteq A$ is called a generating set if there exists a surjective algebra homomorphism $k\langle S\rangle\rightarrow A.$ If $S$ has no subsets that are generating, $S$ is called a minimal generating set.
\end{definition}

For each chain $\{S_i\}_{i\in I\subseteq\mathbb N},\ S_i\subseteq S_{i+1},$ of generating sets we have that $S=\cap_{i\in I}S_i$ is a generating set so that by Zorn's Lemma, every $k$-algebra $A$ has a minimal generating set $S.$ This $S$ is necessarily a linearly independent set. There is a surjection $\phi:k\langle S\rangle\rightarrow A$ and $k\langle S\rangle/\ker\phi\simeq A.$ If the ideal $\ker\phi$ is finitely generated, we say that $A$ is finitely presented.

Let $A$ be a countably generated $k$-algebra, generated over a countable set $S$. Then $A$ is a vector space over $k$ spanned by all monomials $B$ which are products of elements in $S.$ This says $$B=\{s_1^{n_1}s_2^{n_2}\cdots s_m^{n_m}|s_i\in S,\ n\in\mathbb N,\ 1\leq i\leq m\}.$$ We choose an admissible ordering on the basis $B$ consisting of the monomials in $k\langle S\rangle,$ resulting in the relation morphism from Definition \ref{relmorphdef}, $$d:k\langle S\rangle\rightarrow k\langle S\rangle.$$
For each basis element in a basis $\{y_i\}_{i\in I}$ for $k\langle S\rangle/\im d$ we define the polynomial $f_i=\sum_{n\in B}y_i^\ast(n)n$ and consider the ideal $F=(f_i;i\in I).$ Then by Lemma \ref{rellemma} we have that $$A\simeq k\langle S\rangle/F.$$

\section{Formal Algebras}\label{formalgsect}

The free formal associative algebra in $n$ variables over the field $k$ is the tensor-algebra 
$$F=\prod_{i=0}^\infty T^i(V).$$ We use the notation
$$k\langle\langle n\rangle\rangle=k\langle\langle x_1,\dots,x_n\rangle\rangle\overset\rho\rightarrow k,$$ and we let $\mathfrak m=\ker\rho.$ With respect to the choice of base-point, we see that $F$ is augmented over $k$ by sending $x_i$ to $0,$ $1\leq i\leq n.$
Interpreting $\mathfrak m^0=k\langle n\rangle,$ we  can write the vector space $$k\langle\langle n\rangle\rangle=\prod_{i=0}^\infty \mathfrak m^i/\mathfrak m^{i+1}.$$

An algebra $A$ over $k$ with minimal generating set $S$ is a formal $k$-algebra if there exists a $k$-linear surjective homomorphism 
$$\sigma:k\langle\langle S\rangle\rangle=\prod_{i=0}^\infty \mathfrak m^i/\mathfrak m^{i+1}\twoheadrightarrow A$$
so that the vector-space $A$ is $A=\sum_{i=1}^\infty\sigma(\mathfrak m^i/\mathfrak m^{i+1}).$ For a monomial $m\in k\langle S\rangle$ we will refer to $\sigma(m)$ as a monomial in $A,$ and we note that the set of monomials in $A$ spans this vector-space. Choose successively a monomial basis for $\oplus_{i=1}^n\sigma(\mathfrak m^i/\mathfrak m^{i+1}).$ Every element in $A$ can be written as a formal sum of the elements in the successive basis, but we do not claim that this is a monomial basis for $A=\sum_{i=1}^\infty\sigma(\mathfrak m^i/\mathfrak m^{i+1}).$

This choice of successive basis defines the relation morphism $$d: k\langle\langle S\rangle\rangle\rightarrow k\langle\langle S\rangle\rangle$$ as in Definition \ref{relmorphdef}. It is important to see that $\sigma$ is an algebra homomorphism, but $d$ might not be more than a $k$-linear map.

Next, we will use a choice of successive basis and resulting relation morphism $d$ above to give a presentation of the algebra $A$ with finite generator set $S,\ |S|=n.$ 

\vskip0,2cm

Let $\overline H_0=k,\ \overline H_1=k\langle x_1,\dots,x_n\rangle/\mathfrak m^2.$ Put $\overline B_0=\{1\},\ B_1=\{x_1,\dots,x_n\},\overline B_1=\overline B_0\cup B_1.$ Then $\overline B_0$ is a basis for $\overline H_0,$ $B_1$ is a basis for $\mathfrak m/\mathfrak m^2$ and $\overline B_1$ is a basis for $\overline H_1.$ We have the following diagram,
$$\xymatrix{k\langle\langle S\rangle\rangle\ar[r]^\sigma\ar[d]&A\ar[d]\\\overline H_1\ar[r]^{\sigma_1}&A/\sigma(\mathfrak m^2),}$$

and by the minimality of $S$ the induced homomorphism $\sigma_1$ is an isomorphism. 
\vskip0,2cm
Put $H_2=k\langle x_1,\dots,x_n\rangle/\mathfrak m^3\overset{\pi_2'}\rightarrow\overline H_1$ and let $B_2'=\{x_ix_j|1\leq i,j\leq n\}.$ Then $B_2'$ is a basis for $\ker\pi_2'/\mathfrak m^3$ and $\overline B_2'=\overline B_1\cup B_2'$ is a basis for $H_2.$ This says that for each monomial $s\in H_2$ we have a unique relation $$s=\sum_{t\in\overline B_2'}\beta_{s,t}t.$$ For each $t\in\overline B_2'$ let 

$$\langle S,t\rangle=\sum_{s\in\overline B_2'}\sum_{\begin{tiny}\begin{matrix}t_1\cdot t_2=s\\t_1,t_2\in \overline B_1\end{matrix}\end{tiny}}\beta_{s,t}t_1t_2\in k\oplus\mathfrak m/\mathfrak m^2\oplus\mathfrak m^2/\mathfrak m^3\subseteq k\langle S\rangle/\mathfrak m^3.$$

Let $\{y_i\}_{i\in I}$ be a basis for $k\langle\langle S\rangle\rangle/\im d$ where $I$ is an index set. Consider the following element in $k\langle\langle S\rangle\rangle/\im d\otimes H_2:$
$$\begin{aligned}\sum_{n\in\overline B_2'}\langle S,n\rangle\otimes n&=\sum_{n\in\overline B_2'}(\sum_{i\in I}y_i^\ast(\langle S,n\rangle)y_i)\otimes n=\sum_{n\in\overline B_2'}(\sum_{i\in I}y_i^\ast(\langle S,n\rangle)y_i\otimes n)\\
&=\sum_{i\in I}y_i\otimes(\sum_{n\in\overline B_2'}y_i^\ast(\langle S,n\rangle)n).\end{aligned}$$

For each $i\in I$ we put $$f_i^2=\sum_{t\in\overline B_2'}y_i^\ast\langle S,t\rangle t\in H_2$$ where $y_i^\ast$ denotes the dual of $y_i\in k\langle\langle S\rangle\rangle.$ Let $F^2=(f_i^2:i\in I)$ and let $$\overline H_2=H_2/F^2=k\langle x_1,\dots,x_n\rangle/(\mathfrak m^3+F^2)\overset{\pi_2}\rightarrow\overline H_1$$ and choose a basis $B_2\subseteq B_2'$ for $\ker\pi_2.$ Then $\overline B_2=\overline B_1\cup B_2$ is a basis for $\overline H_2$ and we have a unique relation in $\overline H_2.$ For each $s\in H_2,\ s=\sum_{t\in\overline B_2}\beta_{s,t}t.$  We now have the diagram

$$\xymatrix{k\langle\langle S\rangle\rangle\ar[r]^\sigma\ar[d]&A\ar[d]\\\overline H_2\ar[r]^{\sigma_2}&A/\sigma(\mathfrak m^3)}$$ and $\sigma_2$ is a $k$-linear isomorphism because of the dimensions of the finite dimensional linear spaces. The condition in $\overline H_2$ saying that $f_i^2=0,\ i\in I,$ says that in $k\langle\langle S\rangle\rangle/\im d\otimes\overline H_2$ we have

$$\begin{aligned}0=\sum_{i\in I}y_i\otimes f_i^2&=\sum_{i\in I}y_i\otimes\sum_{t\in\overline B_2'}y_i^\ast\langle S,t\rangle t=\sum_{t\in\overline B_2'}\langle S,t\rangle\otimes t=\sum_{t\in\overline B_2'}\langle S,t\rangle\otimes\sum_{u\in\overline B_2}\beta_{t,u}u\\
&=\sum_{u\in\overline B_2}(\sum_{t\in\overline B_2'}\beta_{t,u}\langle S,t\rangle)\otimes u=\sum_{u\in\overline B_2}d(\sigma(u))\otimes u.\end{aligned}$$

For each $u\in B_2$ we have that $$d(\sigma(u))=\sum_{t\in\overline B_2'}\beta_{t,u}\langle S,t\rangle.$$ 

\vskip0,2cm

For $N\geq 2,$ assume that $\overline H_{N}=k\langle\langle S\rangle\rangle/(F^N+\mathfrak m^{N+1})$ has been constructed together with monomial bases $B_{N},\overline B_N.$  Put $$H_{N+1}=k\langle x_1,\dots,x_n\rangle/(\mathfrak m^{N+2}+\mathfrak m F^N+F^N\mathfrak m)\overset{\pi_{N+1}'}\rightarrow\overline H_N$$ and notice that $\ker\pi_{N+1}'$ is spanned by the set of monomials $M=x_i\cdot\overline B_N\cup\overline B_N\cdot x_i,\ 1\leq i\leq n.$ We can write 

$$\begin{aligned}\ker\pi_{N+1}'&=\mathfrak m^{N+1}/(\mathfrak m^{N+2}+\mathfrak m^{N+1}\cap(\mathfrak m F^N+F^N\mathfrak m))\bigoplus F^N/(\mathfrak m F^N+F^N\mathfrak m)\\&=I_{N+1}\oplus F^N/(\mathfrak m F^N+F^N\mathfrak m).\end{aligned}$$

Choose a monomial basis $B_{N+1}'\subseteq M$ for $I_{N+1}.$ By definition the generating polynomials $f_i, i\in I$ are linearly independent such that when $\overline B_{N+1}'=\overline B_N\cup B_{N+1}'$ then $\overline B_{N+1}'\cup \{f_i:i\in I\}$ is a basis for $H_{N+1}.$ This says that every monomial $s\in H_{N+1}$ can be uniquely written 
$$s=\sum_{t\in\overline B_{N+1}'}\beta_{s,t}t+\sum_{i\in I}\beta_{s,i}f_i.$$

For each $t\in B_{N+1}'$ let $$\langle S,t\rangle=\sum_{s\in\overline B_{N+1}'}\sum_{\begin{tiny}\begin{matrix}t_1\cdot t_2=s\\t_1,t_2\in \overline B_N\end{matrix}\end{tiny}}\beta_{s,t}d(\sigma(t_1t_2))\in\oplus_{i=0}^{N}\mathfrak m^i/\mathfrak m^{i+1}\subseteq k\langle\langle S\rangle\rangle/\mathfrak m^{N+1}.$$ Then consider the following element in $k\langle\langle S\rangle\rangle/\im d\otimes H_{N+1}:$

$$\begin{aligned}\sum_{n\in\overline B_{N+1}'}\langle S,n\rangle\otimes n&=\sum_{n\in\overline B_{N+2}'}(\sum_{i\in I}y_i^\ast(\langle S,n\rangle)y_i)\otimes n=\sum_{n\in\overline B_{N+1}'}(\sum_{i\in I}y_i^\ast(\langle S,n\rangle)y_i\otimes n)\\
&=\sum_{i\in I}y_i\otimes(f_i^N+\sum_{n\in B_{N+1}'}y_i^\ast(\langle S,n\rangle)n).\end{aligned}$$

For each $i\in I$ put $$f_i^{N+1}=f_i^N+\sum_{t\in B_{N+1}'}y_i^\ast(\langle S,t\rangle)t$$ and define $$\overline H_{N+1}=H_{N+1}'/(f_i:i\in I)=k\langle\langle S\rangle\rangle/F_{N+1}\overset{\pi_{N+1}}\rightarrow\overline H_N.$$ Choose a monomial basis $B_{N+1}$ for $\ker\pi_{N+1}$ so that $\overline B_{N+1}=B_{N+1}\cup\overline B_N$ is a monomial basis for $\overline H_{N+1}.$ We have the diagram

$$\xymatrix{k\langle\langle S\rangle\rangle\ar[r]^\sigma\ar[d]&A\ar[d]\\\overline H_{N+1}\ar[r]^{\sigma_{N+1}}&A/\sigma(\mathfrak m^{N+2})}$$ and $\sigma_{N+1}$ is a $k$-linear isomorphism because of the dimensions of the finite dimensional linear spaces. Notice the condition in $\overline H_{N+1}$ saying that $f_i^{N+1}=0,\ i\in I,$ says that for each $u\in B_{N+1}$ we have that $$d(\sigma(u))=-\sum_{t\in\overline B_2'}\beta_{t,u}\langle S,t\rangle.$$

By induction,  it follows that $\underset{\underset{n\geq 1}\longleftarrow}\lim\ \overline H_{N+1}\simeq \hat A.$


\section{Matrix Polynomial Algebras}

Let $V=\{V_{ij}\}_{1\leq i,j\leq r}$ be vector spaces with basis $\{x_{ij}(l)\}_{1\leq l\leq n_{ij}}\subset V_{ij}.$ We put $T^i_{k^r}(V)=(V_{ij})^{\otimes_{k^r}i}$ for $i>0$ and $T^0_{k^r}=k^r.$

We define the (certainly not) free matrix polynomial algebra by the following:

\begin{definition} The Free $r\times r$ Matrix Polynomial Algebra in $N=(n_{ij})$ variables is the algebra $$k\langle N\rangle=\oplus_{i\geq 0}T_{k^r}^i(V).$$ 
\end{definition}

We see that this is a $k^r$-algebra which is augmented over $k^r$ by sending $x_{ij}(l)$ to $0$ for all $i,j,l,$ so that $k\langle N\rangle$ fits in the diagram 

$$\xymatrix{k^r\ar[r]^\iota\ar[dr]_{\id}&k\langle N\rangle\ar[d]^\rho\\&k^r.}$$

We let $\mathfrak m=\ker\rho$ and can copy the algorithms from Sections \ref{AlgebrasSection} and \ref{formalgsect} word by word.

\begin{lemma} Let $A$ be a $k^r$-algebra augmented over $k^r,$ and let $\mathfrak a\subseteq A$ be an ideal. Then $\mathfrak a=\oplus_{1\leq i,j\leq r}\mathfrak a_{ij}\subseteq A_{ij}$ where $\mathfrak a_{ij}=e_i\mathfrak a e_j.$
\end{lemma}

\begin{proof} This follows trivially because $\mathfrak a$ is a (two-sided) ideal and so contains the subideal generated by de idempotents.
\end{proof}

\section{GMMP-algebras}\label{GMMPDef}

\begin{definition}
A GMMP-algebra consists of the data $(V,W,\cup,d)$ where $V,W$ are vector spaces  over a field $k$ and  $\cup:V\otimes_k V\rightarrow W,\ d:V\rightarrow W$ are linear maps. A morphism from $(V,W,\cup,d)$ to $(V',W',\cup',d')$ consists of a pair of linear maps $\phi:V\rightarrow V',\ \psi:W\rightarrow W'$ such that the following diagrams commutes:

$$\xymatrix{V\otimes_k V\ar[r]^{\cup}\ar[d]_{\phi\otimes_k\phi}&W\ar[d]^{\psi}\\V'\otimes_k V'\ar[r]^{\cup'}&W',}
\xymatrix{V\ar[r]^d\ar[d]_{\phi}&W\ar[d]^{\psi}\\V'\ar[r]^{d'}&W'.}$$
\end{definition}

For a GMMP-algebra $L=(V,W,\cup,d),$ if $X=\{x_1,\dots,x_n\}\subseteq V$ is a linearly independent set, we let $B=\{x_{i_1}\cup\cdots\cup x_{i_l}|\ l\geq 0,\ 1\leq i_d\leq n, 1\leq d\leq l\}.$ Then $B$ is the set of all monomials in $W$ under the product $\cup.$

\begin{definition} A GMMP-algebra $L=(V,W,\cup,d)$ is called polynomial with respect to a linearly independent set $X\subseteq V$ if the set $B$ is a basis for a vector space $W_s\supseteq\im d.$ We write $\widehat L=L$ and say that $\widehat L$ is formal with respect to $X,$ if there is a vector space $\widehat W_s,$ $\im d\subseteq\widehat W_s\subseteq W$ which is formal with respect to  $B.$
\end{definition}

In the following, given a finitely generated $k$-algebra $A=k\langle x_1,\dots,x_n\rangle/\mathfrak a,$ we will use $\mathfrak m=(x_1,\dots,x_n)\subseteq A,$ and we will always use the name \emph{monomial} for an element $x_1^{m_1}\cdots x_n^{m_n}\in A,\ m_i\in\mathbb N$ (where we assume $0\in\mathbb N$). A monomial basis for some subvector space in $A$ means a basis consisting of monomials. Finally, for $v_1,v_2\in V$ we write $v_1v_2=v_1\cup v_2$ for short.

\begin{definition}\label{Algorithm}

Let $L=(V,W,\cup,d)$ be a GMMP-algebra and assume that $$X=\{v_1,\dots,v_n\}\subseteq V$$ is a linearly independent subset. 

If $L$ is polynomial with respect to $X,$ let $F=k\langle X\rangle.$ Then the set  of monomials in $F$ is a basis of this vector space.   For a monomial $x=x_1^{n_1}x_2^{n_2}\cdots x_r^{n_r}$ where $x_i\in V,n_i\in\mathbb N,\ 1\leq i\leq r$ we let $m(x)=x_1^{n_1}\cup\cdots\cup x_r^{n_1}.$ We use the notation $x^n=\cup_{i=1}^n x.$ Let $W_s\subseteq W$ be the subspace spanned by all $m(x), x\in B, $ and choose a basis $\{y_i\}_{i\in I}$ for $W_s/\im d\cap W_s.$ Let $$f_i=\sum_{x\in B}y_i^\ast(m(x))x\in F, i\in I.$$
We define the algebra of $A(L)$ to be the algebra $F/(f_i;i\in I).$

If $W$ is formal with respect to $X$, let $\hat F=k\langle\langle X\rangle\rangle.$

We use the procedure in Section \ref{formalgsect} to define a formal algebra $\hat{A}(L)$
inductively. Let $$\overline H_0=k,\ \overline H_1=k\langle x_1,\dots,x_n\rangle/\mathfrak m^2,\ \mathfrak m=(x_1,\dots,x_n).$$ Put $\overline B_0=\{1\},\ B_1=\{x_1,\dots,x_n\},\overline B_1=\overline B_0\cup B_1.$ Then $\overline B_0$ is a basis for $\overline H_0,$ $B_1$ is a basis for $\mathfrak m/\mathfrak m^2$ and $\overline B_1$ is a basis for $\overline H_1.$ We let $v_0=1$ and $v_{x_i}=v_i,\ 1\leq i\leq n.$

Put $H_2=k\langle x_1,\dots,x_n\rangle/\mathfrak m^3\overset{\pi_2'}\rightarrow\overline H_1$ and let $B_2'=\{x_ix_j|1\leq i,j\leq n\}.$ Then $B_2'$ is a basis for $\ker\pi_2'/\mathfrak m^3$ and $\overline B_2'=\overline B_1\cup B_2'$ is a basis for $H_2.$ This says that for each monomial $s\in H_2$ we have a unique relation $$s=\sum_{t\in\overline B_2'}\beta_{s,t}t.$$ 

 
For each $t\in\overline B_2'$ let 

$$\langle X,t\rangle=\sum_{s\in\overline B_2'}\sum_{\begin{tiny}\begin{matrix}t_1\cdot t_2=s\\t_1,t_2\in \overline B_1\end{matrix}\end{tiny}}\beta_{s,t}v_{t_1}v_{t_2}\in k\oplus W.$$

Let $\{y_i\}_{i\in I}$ be a basis for $W/\im d$ where $I$ is an index set. Let $y_i^\ast$ denote the dual of $y_i$ and consider the following element in $W\otimes H_2:$
$$\begin{aligned}\sum_{n\in\overline B_2'}\langle X,n\rangle\otimes n&=\sum_{n\in\overline B_2'}(\sum_{i\in I}y_i^\ast(\langle X,n\rangle)y_i)\otimes n=\sum_{n\in\overline B_2'}(\sum_{i\in I}y_i^\ast(\langle X,n\rangle)y_i\otimes n)\\
&=\sum_{i\in I}y_i\otimes(\sum_{n\in\overline B_2'}y_i^\ast(\langle X,n\rangle)n).\end{aligned}$$

For each $i\in I$ we put $$f_i^2=\sum_{t\in\overline B_2'}y_i^\ast\langle X,t\rangle t\in H_2.$$ Let $F^2=(f_i^2:i\in I)$ and let $$\overline H_2=H_2/F^2=k\langle x_1,\dots,x_n\rangle/(\mathfrak m^3+F^2)\overset{\pi_2}\rightarrow\overline H_1$$ and choose a basis $B_2\subseteq B_2'$ for $\ker\pi_2.$ Then $\overline B_2=\overline B_1\cup B_2$ is a basis for $\overline H_2$ and we have a unique relation in $\overline H_2.$ For each $s\in \overline H_2,\ s=\sum_{t\in\overline B_2}\beta_{s,t}t.$   The condition in $\overline H_2$ saying that $f_i^2=0,\ i\in I,$ says that for each $u\in B_2$ we can choose an element $v_u\in V$ such that
$$d(v_u)=-\sum_{t\in\overline B_2'}\beta_{t,u}\langle X,t\rangle.$$ 

\vskip0,2cm

For $N\geq 2,$ assume that $\overline H_{N}=k\langle X\rangle/(F^N+\mathfrak m^{N+1})$ has been constructed together with monomial bases $B_{N},\overline B_N.$  Put $$H_{N+1}=k\langle x_1,\dots,x_n\rangle/(\mathfrak m^{N+2}+\mathfrak m F^N+F^N\mathfrak m)\overset{\pi_{N+1}'}\rightarrow\overline H_N$$ and notice that $\ker\pi_{N+1}'$ is spanned by the set of monomials $M=x_i\cdot\overline B_N\cup\overline B_N\cdot x_i,\ 1\leq i\leq n.$ We can write 

$$\begin{aligned}\ker\pi_{N+1}'&=\mathfrak m^{N+1}/(\mathfrak m^{N+2}+\mathfrak m^{N+1}\cap(\mathfrak m F^N+F^N\mathfrak m))\bigoplus F^N/(\mathfrak m F^N+F^N\mathfrak m)\\&=I_{N+1}\oplus F^N/(\mathfrak m F^N+F^N\mathfrak m).\end{aligned}$$

Choose a monomial basis $B_{N+1}'\subseteq M$ for $I_{N+1}.$ By definition the generating polynomials $f_i, i\in I$ are linearly independent such that when $\overline B_{N+1}'=\overline B_N\cup B_{N+1}'$ then $\overline B_{N+1}'\cup \{f_i:i\in I\}$ is a basis for $H_{N+1}.$ This says that every monomial $s\in H_{N+1}$ can be uniquely written 
$$s=\sum_{t\in\overline B_{N+1}'}\beta_{s,t}t+\sum_{i\in I}\beta_{s,i}f_i.$$

For each $t\in B_{N+1}'$ let $$\langle X,t\rangle=\sum_{s\in\overline B_{N+1}'}\sum_{\begin{tiny}\begin{matrix}t_1\cdot t_2=s\\t_1,t_2\in \overline B_N\end{matrix}\end{tiny}}\beta_{s,t}v_{t_1}v_{t_2}\in\oplus_{i=0}^{N}\mathfrak m^i/\mathfrak m^{i+1}\subseteq k\langle S\rangle/\mathfrak m^{N+1}.$$ Then consider the following element in $k\langle X\rangle/\im d\otimes H_{N+1}:$

$$\begin{aligned}\sum_{n\in\overline B_{N+1}'}\langle X,n\rangle\otimes n&=\sum_{n\in\overline B_{N+2}'}(\sum_{i\in I}y_i^\ast(\langle X,n\rangle)y_i)\otimes n=\sum_{n\in\overline B_{N+1}'}(\sum_{i\in I}y_i^\ast(\langle X,n\rangle)y_i\otimes n)\\
&=\sum_{i\in I}y_i\otimes(f_i^N+\sum_{n\in B_{N+1}'}y_i^\ast(\langle X,n\rangle)n).\end{aligned}$$

For each $i\in I$ put $$f_i^{N+1}=f_i^N+\sum_{t\in B_{N+1}'}y_i^\ast(\langle X,t\rangle)t$$ and define $$\overline H_{N+1}=H_{N+1}'/(f_i:i\in I)=k\langle X\rangle/F_{N+1}\overset{\pi_{N+1}}\rightarrow\overline H_N.$$ Choose a monomial basis $B_{N+1}$ for $\ker\pi_{N+1}$ so that $\overline B_{N+1}=B_{N+1}\cup\overline B_N$ is a monomial basis for $\overline H_{N+1}.$ The condition in $\overline H_{N+1}$ saying that $f_i^{N+1}=0,\ i\in I,$ says that for each $u\in B_{N+1}$ we can choose an element $v_u\in V$ such that 
$$d(v_u)=-\sum_{t\in\overline B_2'}\beta_{t,u}\langle X,t\rangle.$$

Now we define $$\hat A(L)=\underset{\underset{n\geq 1}\longleftarrow}\lim\ \overline H_{N+1}.$$

\end{definition}

\begin{lemma} If there exists $n\in\mathbb N$ such that the $v_u$'s chosen above can be chosen such that $v_u$ is zero in all degrees less that $\lceil\frac{n}{2}\rceil,$ then $L\llbracket X\rrbracket$ is algebraic of degree at most $n+1.$
\end{lemma}

\begin{theorem}  There is a fully faithful functor between the category of $k$-algebras with countable generator set $X$ and the category of GMMP-algebras with linearly independent set $X,$
$\hat{L}:\cat{Alg}_k(X)\rightarrow\cat{GMMP}_k(X),$ with a left inverse $$\hat A_X:\cat{GMMP}_k(X)\rightarrow\cat{Alg}_k(X),$$ i.e. such that $\hat A_X(\hat L_X(\hat A_X))\simeq\hat A.$
\end{theorem}

\begin{proof} The construction in Section \ref{AlgebrasSection} gives the GMMP-algebra with $S\subseteq A=V$ and $W=k\langle S\rangle.$ This also give the reverse equivalence.
\end{proof}

\begin{remark}\label{uniquerem} Note that $\hat A(\hat L)$ is unique up to non-unique isomorphism. This is due to the choices of monomial bases at each stage.
\end{remark}

\begin{example} We give a polynomial GMMP-algebra $L=(V,W,\cup,d)$ with linearly independent set $v_1,v_2$ by the following.

$$\left(\begin{matrix}
v_0&v_1&v_2&v_3&v_4
\\1&x&y&x^2&xy\\
v_0+v_3+v_6&v_1+v_7+v_{10}&v_2+v_8+v_{14}&0&v_4\end{matrix}\right.
$$

$$\left.\begin{matrix}
v_5&v_6&v_7&v_8&v_9&v_{10}&v_{11}&v_{12}&v_{13}&v_{14}
\\yx&y^2&x^3&x^2y&xyx&xy^2&yx^2&yxy&y^2x&y^3\\
v_5&0&0&v_8+v_{11}&0&v_{10}+v_{13}&0&0&0&0\end{matrix}\right).
$$
 
The matrix has the following interpretation: The vector space $V$ is infinite dimensional, with linearly independent subset $X=\{x,y\}.$ The vector space $W$ is equal to $V$ and both have basis $\{1,x,y,x^2,xy,yx,y^2,..\}$ and further on, where $mn=m\cup n$ for $m,n\in V=W.$ This gives the full description of a GMMP-algebra, and by Definition \ref{Algorithm} the computation amounts to compute the relation morphism $d_{WW},$ choose a basis $\{y_i\}_{i\in\mathbb N}$ for $W/\im d_{WW}$ and then put $$A(L)=k\langle x,y\rangle/(f_i)_{i\in\mathbb N}$$ where $f_i=\sum_{\underline m}y_i^\ast(\underline m)\underline m$ where $\underline m$ runs through all the monomials in $x,y.$ In this example we find $$A(L)=k\langle x,y\rangle/(1-x^2-y^2).$$
\end{example}

\section{Completions of Associative Algebras}
In commutative algebra, we use the topological concept of completions to form the $\mathfrak a$-adic completion $\hat A^{\mathfrak a}$ of a commutative ring in an ideal $\mathfrak a\subset A.$ We observe that when we take the completion in a maximal ideal $\mathfrak m\subset A$ we obtain a local ring $\hat A$ which by the universal property of localization contains the image of the local ring $A_{\mathfrak m}.$ This also works for completions in prime ideals $\mathfrak p\subset A,$ this is because $A_{\mathfrak p}$ is a local ring. Thus $A_f\hookrightarrow\hat A_{\mathfrak p}$ for each  $f\notin\mathfrak p$ and so $A_{\mathfrak p}\hookrightarrow\hat A_{\mathfrak p}.$ For our constructions it is sufficient to consider completion in simple modules.

Let $A$ be an associative $k$-algebra, let $M=\{M_i\}_{i=1}^r$ be a set of $r$ simple right $A$-modules.Put $$V=(\hmm_k(A,\hmm_k(V_i,V_j))),\ W=(\hmm_k(A^{\otimes 2},\hmm_k(V_i,V_j))).$$ Let $d:V\rightarrow W$ be given by $d(\phi)(a_1\otimes a_2)=a_1\phi(a_2)-\phi(a_1)a_2,$ and $\cup:V\otimes_k V\rightarrow W$ by $\phi\cup\psi(a\otimes b)=\phi(a)\circ\phi(b).$ Then $L=(V,W,\cup,d)$ is a GMMP-algebra, and it is proved in \cite{ELS17} that if we let $X=(\{x_{ij}(l)\}_{1\leq l\leq\l_{ij}})_{1\leq i,j\leq r}\subset V$ be representatives for a basis for $V/\ker d,$ then the formal GMMP-algebra generated by $X.$

\begin{theorem}
With the notation above, let $$\hat H(M)=\hat A(X\subset V,W,\cup,d)$$ and put $\hat A_M=\enm_k(\hat H(M)\otimes_{k^r}\oplus_{i=1}^r M).$ Then there exists an algebra homomorphism $\iota:A\rightarrow\hat A_M$ commuting in the diagram 
$$\xymatrix{A\ar[r]^-\iota\ar[dr]_{\oplus\rho_i}&\hat A_M\ar[d]^\delta\\
&\oplus_{i=1}^r\enm(M_i)}
$$ where $\rho_i:A\rightarrow\enm_k(M_i)$ is the structure morphism and $\delta$ is the diagonal.
\end{theorem}

\begin{proof} This can be proven directly by working in the Hochschild complex, and then it is proved in \cite{ELS17}.
\end{proof}

\begin{definition} Let $A$ be an associative algebra and let $M=\{M_1,\dots,M_r\}$ be a set of $r$ simple right $A$-modules. Then the completion of $A$ in $M$ is defined as $$\hat A_M=\enm_k(\hat H(M)\otimes_{k^r}\oplus_{i=1}^r M),$$ where
$\hat H(M)=\hat A(X\subset C^1(A,M),C^2(A,M),\cup,d).$
\end{definition}

Notice that by Remark \ref{uniquerem}, $\hat H(M),$ and thereby $\hat A_M$ is determined upto (non-unique)isomorphism.

\begin{proposition} If $A$ is a commutative $k$-algebra and $\mathfrak m\subset A$ a maximal ideal, let $M=A/\mathfrak m.$Then we have that $\hat A_{M}\simeq \hat A_{\mathfrak m}.$
\end{proposition}

\begin{proof} This follows directly from the fact that the representatives $X$ is a generating set of the maximal ideal $\mathfrak m\subset A,$ and so the algebra of the GMMP-algebra is formed by all power-series in variables from $X$ which coincides with the completion in the commutative situation.
\end{proof}

\section{Deformations of modules}\label{defsection}

Let $A$ be an associative $k$-algebra, and let $M=\{M_1,\dots, M_r\}$ be a set of right $A$-modules. We let $\cat{Art}_k$ be the category of $r$-pointed $k$-algebras, the objects being $k$-algebras fitting in the diagram 
$$\xymatrix{k^r\ar[r]^\iota\ar[dr]_\id&A\ar[d]^\rho\\&k^r,}$$
and such that $\ker^n\rho=(\ker\rho)^n=0$ for some $n\geq 1.$ The morphisms are algebra homomorphisms commuting with $\iota$ and $\rho.$ The category of formal $r$-pointed $k$-algebras $\widehat{\cat{Art}}_k$ consists of $r$-pointed algebras $\hat A$ such that for all $n\geq 0,\ A/\ker^n\rho$ is an object in $\cat{Art}_k.$ This says that the objects in $\widehat{\cat{Art}}_k$ are projective limits of objects in $\cat{Art}_k.$

\begin{definition} A deformation of the set $M=\{M_1,\dots, M_r\}$ to the formal $r$-pointed algebra $S$ is an $S\otimes_{k^r}A$-module $M_S$ such that $M_S\simeq S\otimes_{k^r}(\oplus_{i=1}^rM_i)$ as $S$-module (which is proved to be equivalent to $M_S$ being $S$-flat in \cite{ELS17}), and such that $k^r\otimes_S M_S\simeq \oplus_{i=1}^rM_i.$ This defines the (covariant) noncommutative deformation functor
$$\defm M:\widehat{\cat{Art}}_k\rightarrow\cat{Sets}$$
\end{definition}

\begin{proposition}\label{prorepprop}
$\hat A_M$ is a prorepresenting hull for the deformation functor $\defm M.$
\end{proposition}

\begin{proof} This is proved by direct calculation in \cite{ELS17} by using two different complexes
$(C^\cdot(A,M),d)$ computing $(\ext^i_A(M_i,M_j))_{1\leq i,j\leq r}.$
\end{proof}

\subsection{Deformation of GMMP-algebras}
Let $L=(X\subseteq V,W,\cup,d)$ be a GMMP algebra, and assume that $X$ is finite.  The data of a GMMP-algebra consists of a linear homomorphism $\cup:V\otimes_k V\rightarrow W,$ and a linear homomorphism $d:V\rightarrow W.$ We assume that we can choose ordered bases $B_V,V_W$ for $V,W$ respectively, and with this choice of bases the GMMP-data are determined by their matrices with respect to the chosen bases. This says that the pair of (possibly infinite dimensional) matrices $$(\cup,d)\in M(B_V\times B_V, B_W)\times M(B_V,B_W)$$ represents a GMMP-algebra.This is a point in $$k^{((B_V\times B_V)B_W+B_VB_W)}=k^B$$ which corresponds to a closed point in $\mathbb A_k^B.$ Different choices of bases results in different points representing isomorphic GMMP-algebras. Thus the orbits under the natural action of $G=\gl(B_V)\times\gl(B_W)$ on $\mathbb A_k^B$ are the isomorphism classes of the GMMP-algebras. Thus we have a commutative ring $A$ on which the abelian group $G$ acts, and each isomorphism class of a GMMP-algebra corresponds to a $A[G]$-module where $A[G]$ is the skew polynomial algebra. Then we can deform any GMMP-algebra as an $A[G]$-module. This is described in the book \cite{S23}.

\section{Deformation of Algebras}

A deformation of any object in any category is a particular kind of lifting of the object in a moduli. The form of the deformation, e.g. the flatness in the situation of modules, is a necessity for existence of a particular kind of moduli.

Given an associative algebra $A.$ Then we can prove that the deformations of $L(A)$ stays in the class of moduli of algebras, that is, if $L_H$ is a deformation of $L(A)$ to $H$, then $L_H=L(A')$ for some algebra $A'$ and thus $A'$ is a deformation of $A.$ 

\begin{lemma} A deformation of $L(A)$ to $S\in\widehat{\cat{Art}}_k$ is a on the form $L(A')$ with $A'$ a deformation of $A$ to $S.$ 
\end{lemma}

\begin{proof} This follows directly from the requirement of flatness of deformations of modules. All relations must be conserved when deforming.
\end{proof}

\begin{corollary} To give a deformation of an algebra $A$ to $S$ is equivalent to give a deformation of $L(A)$ to $S.$
\end{corollary}

\begin{proof} The only thing left to prove is that if $A'$ is a deformation of $A,$ then $L(A')$ is a deformation of $L(A).$ This follows from the definition on the generators of $A.$
\end{proof}

\end{document}